\documentclass{article}
\usepackage{authblk}
\usepackage{amsfonts,amsmath,amssymb,amsthm,hyperref}
\usepackage{xcolor, soul}
\sethlcolor{red}
\usepackage{graphicx} 
\usepackage{tikz}
\usepackage{comment}

\newtheorem{theorem}{Theorem}
\newtheorem{lemma}{Lemma}

\newtheorem{prop}{Proposition}
\newtheorem{coroll}{Corollary}
\theoremstyle{definition}

\newtheorem{example}{Example}[section]

\renewcommand\det{\operatorname{Fix}}
\newcommand\fix{\operatorname{Fix}}
\newcommand\dist{\operatorname{Dist}}

\newcommand\fdist{\operatorname{Fdist}}
\newcommand\fcost{\operatorname{Fcost}}
\newcommand\tnk{T^n_k}

\usepackage{todonotes}

\definecolor{korange}{RGB}{255,204,102}

\begin{document}

\title{Paint cost spectrum of perfect $k$-ary trees}

\author[1,2]{Sonwabile Mafunda\thanks{Financial support from Soka University of America (Varvak's RDA), USA and from the DSI-NRF Centre of Excellence in Mathematical and Statistical Sciences (CoE-MaSS), South Africa  is greatly acknowledged.}}
\author[1]{Jonathan L. Merzel}
\author[1]{K. E. Perry}
\author[1]{Anna Varvak}
%\date{May 2023}

\affil[1]{Soka University of America\\
USA}
\affil[2]{University of Johannesburg\\
South Africa}
%\affil[3]{National Institute for Theoretical and Computational Sciences, South Africa\\}% (NITheCS)\\South Africa\\}

%\author{ , 
%,
%    \\
%University of Johannesburg}

\maketitle

\begin{abstract}
We determine the paint cost spectrum for perfect $k$-ary trees.

A coloring of the vertices of a graph $G$ with $d$ colors is said to be \emph{$d$-distinguishing} if only the trivial automorphism preserves the color classes. The smallest such $d$ is the distinguishing number of $G$ and is denoted $\dist(G).$ The \emph{paint cost of $d$-distinguishing $G$}, denoted $\rho^d(G)$, is the minimum size of the complement of a color class over all $d$-distinguishing colorings. A subset $S$ of the vertices of $G$ is said to be a \emph{fixing set} for $G$ if the only automorphsim that fixes the vertices in $S$ pointwise is the trivial automorphism. The cardinality of a smallest fixing set is denoted $\fix(G)$. In this paper, we explore the breaking of symmetry in perfect $k$-ary trees by investigating what we define as the \emph{paint cost spectrum} of a graph $G$: $(\dist(G); \rho^{\dist(G)}(G), \rho^{\dist(G)+1}(G), \dots, \rho^{\fix(G)+1}(G))$ and the \emph{paint cost ratio} of $G$, which is defined to be the fraction of paint costs in the paint cost spectrum equal to $\fix(G)$. We determine both the paint cost spectrum and the paint cost ratio completely for perfect $k$-ary trees.

We also prove a lemma that is of interest in its own right: given an $n$-tuple, $n \geq 2$ of distinct elements of an ordered abelian group and $1 \leq k \leq n! -1$, there exists a $k \times n$ row permuted matrix with distinct column sums.

\end{abstract}

\section{Introduction}\label{intro}

 Given a graph $G$, a natural question to ask is, \emph{What kind of symmetries exist in $G$?} One way to investigate this question is to consider how difficult it is to break all symmetries in the graph. In this paper, we unify two symmetry-breaking parameters -- the distinguishing number and the fixing number -- along with the paint cost of each, in what we define as the \textit{paint cost spectrum} of a graph.

All graphs in this paper are finite and simple. A coloring of the vertices of a graph $G$ with $d$ colors is said to be a \emph{$d$-distinguishing coloring} of $G$ if only the trivial automorphism preserves the color classes of $G$. A graph is \emph{$d$-distinguishable} if it has a $d$-distinguishing coloring and the \emph{distinguishing number} of $G$, denoted $\dist(G)$, is the smallest number of colors necessary for a distinguishing coloring of $G$. Graph distinguishing was introduced independently by Albertson and Collins in~\cite{AC1996} and by Babai in~\cite{Ba1977}. 

For a large number of graph families, all but a finite number of members are $2$-distinguishable, for example: hypercubes $Q_n$ with $n\geq 4$~\cite{BC2004}, Cartesian powers $G^n$ for a connected graph $G\ne K_2,K_3$ and $n\geq 2$~\cite{A2005, IK2006,KZ2007}, and Kneser graphs $K_{n:k}$ with $n\geq 6, k\geq 2$~\cite{AB2007}. 
The \emph{cost}, denoted $\rho(G)$, was defined in the context of a  2-distinguishable graph $G$ as the minimum size of a color class over all 2-distinguishing colorings of $G$~\cite{B2008}. 
Some graph families with known or bounded cost are hypercubes~\cite{B2008}, Kneser graphs~\cite{B2013b}, and the Cartesian product between $K_{2^m}$ and a graph without symmetries~\cite{BI2017}.

There are two perspectives from which one can generalize the cost to any $d$-distinguishable graph $G$: the minimum number of vertices in the graph needing to be recolored from a neutral color in order to have that $d$-distinguishing coloring of $G$, or, how far away is $G$ from being $(d-1)$-distinguishable.
Boutin~\cite{B2023} generalizes the cost from the first perspective, by defining the \emph{paint cost} of $d$-distinguishing $G$, denoted $\rho^d(G)$, (equivalently, the \emph{$d$-paint cost} of $G$) as the minimum size of the complement of a color class over all $d$-distinguishing colorings. 
Alikhani and Soltani~\cite{AK21} generalize the cost from the second perspective, by defining the \emph{cost number}, $\rho_d(G)$, as the size of the smallest color class over all $d$-distinguishing colorings of $G$. In this paper, we focus on Boutin's generalization. Note that for $2$-distinguishable graphs, we have that $\rho(G) = \rho^2(G) = \rho_2(G).$

 Boutin observed~\cite{B2023} that if $d \geq \dist(G)$, then $\rho^d(G) \geq \rho^{d+1}(G)$. Thus, the largest $d$-paint cost for a graph $G$ is $\rho^{\dist(G)}(G)$. It turns out that the smallest $d$-paint cost for a graph $G$ is equal to the \emph{fixing number} of $G$~\cite{B2023}. A subset $S \subseteq V(G)$, is said to be a \emph{fixing set} for $G$ if the only automorphism that fixes the elements of $S$
pointwise is the trivial automorphism. Intuitively, as the third author observed in~\cite{BCKLPR2020b}, when we consider automorphisms of the vertices, we can think of a fixing set as a set of vertices that when pinned in place, fix the entire graph. The size of a smallest fixing set of a graph $G$, denoted $\det(G)$, is called the \emph{fixing number} of $G$. The fixing numbers were introduced independently by Harary in~\cite{H2001} and Boutin in~\cite{B2006} and have also been studied in the literature under the names \textit{determining number} and \textit{rigidity index.}

Though distinguishing numbers and fixing numbers were introduced separately, there is a clear relationship between the two: Given a graph $G$, if one assigns to each vertex in a fixing set a unique color from all other vertices in the graph, this is a $(\det(G)+1)$-distinguishing coloring of $G$  with paint cost $\det(G)$. Thus, $\dist(G) \leq \det(G)+1$ and the smallest possible paint cost for $G$ is less than or equal to $\det(G)$.  On the other hand, since in any distinguishing coloring, the complement of a color class is always a fixing set, the paint cost of any coloring is greater than or equal to  $\det(G)$.  Thus, the smallest possible paint cost for a graph $G$ is  $\det(G)$, and this cost is achievable with  $\det(G)+1$ colors.

The introduction of $d$-paint cost makes way for many exciting and interesting questions with regard to distinguishing colorings of graphs. In particular, we are interested in investigating the $d$-paint cost of $d$-distinguishing colorings for all possible $d$. For a graph $G$, we define the \emph{paint cost spectrum} of a graph $G$ to be the list:
\begin{center}
$(\dist(G); \rho^{\dist(G)}(G), \rho^{\dist(G)+1}(G), \dots, \rho^{\det(G)+1}(G)).$
\end{center}

Perhaps contrary to one's intuition, the paint cost spectrum is not necessarily strictly decreasing. In particular, the \emph{frugal distinguishing number} of a graph $G$, denoted $\fdist(G)$, is the smallest $d$ for which $\rho^d(G) = \det(G)$~\cite{B2023}. Thus, if we let $D = \dist(G)$, $F = \fdist(G)$, then the paint cost spectrum of a graph $G$ is the list $(D; \rho^D(G), \rho^{D+1}(G), \dots, \rho^F(G) = \fix(G), \dots, \rho^{\det(G)+1}(G) = \fix(G))$.

Of interest also, of course, is how much of the spectrum consists of paint cost equal to $\fix(G)$, including the entirety of the spectrum itself. Motivated by this question, we introduce the \emph{paint cost ratio} of a graph $G$ and define it to be the ratio:
\begin{center}
$\displaystyle \frac{\fix(G) - \fdist(G) + 2}{\fix(G) - \dist(G) + 2}.$
\end{center}

We observe that this ratio is just the fraction of costs in the paint cost spectrum that are equal to $\fix(G)$. 

In this paper, we investigate the paint cost spectrum and paint cost ratio for perfect $k$-ary trees of depth $n$. For $n \geq 1$, $k \geq 2$, a \emph{perfect $k$-ary tree of depth $n$}, denoted $T_k^n$, is a connected acyclic graph with $k^n$ leaves, where there is one vertex, which we select as the root with degree $k$ whose distance to all the leaves is $n$ and all other internal vertices have degree $k+1$.  Note that, by degree, the root is uniquely determined, and so we can unambiguously speak of the depth of a vertex. We will show the following.

\newtheorem*{pc}{Theorem \ref{pc}}
\begin{pc}
Let $T^n_k$ be a perfect $k$-ary tree of depth $n$ for $n \geq 1$. Then, 
\begin{itemize}
    \item for $n=1,$ the paint cost spectrum is $(k;k-1)$ and the paint cost ratio is $1$,
    \item for $n=2,$ the paint cost spectrum is $(k; k^2-1, k(k-1), \dots, k(k-1))$ and the paint cost ratio is $\displaystyle \frac{k^2-2k+1}{k^2-2k+2}$,
    \item for $n \geq 3$, the paint cost spectrum is $$(k; k^n-1, (k-1)(k^2+1)k^{n-3}, k^{n-1}(k-1), \dots, k^{n-1}(k-1))$$ and the paint cost ratio is $ \displaystyle \frac{k^{n-1}(k-1)-k}{k^{n-1}(k-1) - k+2}$.
\end{itemize}
\end{pc}

The paper is organized as follows. In Section~\ref{prelim}, we provide some necessary background and a discussion on the paint cost spectrum and paint cost ratio. In Sections~\ref{sec-dist},  \ref{sec-fix}, and \ref{sec-middle}, we determine the full paint cost spectrum of all perfect $k$-ary trees, including the distinguishing number, fixing number, and fugal number of all such trees. In Section~\ref{equality}, we include a discussion on graphs where all costs on the paint cost spectrum are equal, addressing a question of Boutin~\cite{B2023}. In particular, we show that for any paint cost greater than or equal to one, there exists an infinite family of trees with equality on the paint cost spectrum. Section~\ref{biglemma} consists of a generalization of a lemma used in an earlier proof which shows that given an $n$-tuple of distinct elements of an ordered abelian group, $n \geq 2$, and an integer $1 \leq k \leq n!-1$, there exists a $k \times n$ row permuted matrix with distinct column sums. Finally, in Section~\ref{future}, we conclude with some future directions and open problems.

\section{Preliminaries}\label{prelim}

In this section, we delve deeper into the paint cost spectrum and provide some tools and definitions we will use throughout the paper.

\subsection{The Paint Cost Spectrum and Ratio}\label{prelim1}

When investigating $d$-distinguishing colorings of a graph $G$ where $d$ can range from $\dist(G)$ to $\fix(G) +1$, it is of interest what paint cost values can be achieved and how these values relate to one another. Thus, we defined the \emph{paint cost spectrum} of a graph $G$ to be the list 
$$(\dist(G); \rho^{\dist(G)}(G), \rho^{\dist(G)+1}(G), \dots, \rho^{\det(G)+1}(G))$$

\noindent so that the values could be studied as a collection. When it causes no confusion to the reader, we sometimes refer to the paint cost spectrum as simply the spectrum.

We note here that while it may not be explicitly necessary to begin the paint cost spectrum by specifying $\dist(G)$, we do so for simplicity of understanding the number of colors to which the list of paint costs are associated. Of course since the final paint cost in the spectrum is $\det(G)$ using $\det(G)+1$ colors, one could count the total number of entries in the spectrum to determine a graph's distinguishing number. However, as in Theorem~\ref{pc}, it is often convenient to represent repeated paint costs in the paint cost spectrum with dots and so it becomes necessary to include a way to determine the number of values in the spectrum. Thus, we begin the spectrum by including the distinguishing number. We further note here that it was shown in~\cite{BI2017} that the cost of 2-distinguishing a graph can be an arbitrarily large multiple of the fixing number, so it follows that the length of the paint cost spectrum can also be arbitrarily large. 

Because the last entry in the paint cost spectrum corresponds to the paint cost of a $(\fix(G)+1)$-distinguishing coloring which uses $\fix(G)$ non-neutral colors and the first paint cost in the spectrum corresponds to the paint cost of a $\dist(G)$-distinguishing coloring, two graphs can only have the same paint cost spectrum if they have the same distinguishing number and the same fixing number. We further observe that it is certainly not always the case that graphs can be uniquely identifiable by their paint cost spectra. For example, any asymmetric graph will have paint cost spectrum $(1;0)$ and any graph with fixing number 1, such as path graphs, will have distinguishing number 2 and paint cost spectrum $(2;1)$. Thus, one can ask whether there exists a graph characterized by its paint cost spectrum.

It turns out, however, that this is not the case: Let $G$ be a connected graph with $V(G) \geq 2$ and let $G'$ be a copy of $G$ with a new vertex $v$ adjacent to every vertex. Then, attach to $v$ a single leaf, $\ell$. Since $\ell$ and $v$ are uniquely identifiable in $G'$ as the only leaf vertex and only vertex adjacent to a leaf, respectively, both can receive the neutral color in any distinguishing coloring of $G'$. Thus, any $d$-distinguishing coloring of $G$ extends to a $d$-distinguishing coloring of $G'$ in the natural way. Similarly, any $d$-distinguishing coloring of $G'$ can be mapped to vertices in $G$.

One big question of interest, however, concerning the paint cost spectrum is whether or not equal paint costs exist aside from $\rho^d(G)$ where $d \geq \fdist(G)$. In other words, can there be integers $i$ and $i+1$, $i+1 < \fdist(G)$, such that $\rho^i(G) = \rho^{i+1}(G)$? And, in particular, is it possible to have $\rho^{\dist(G)}(G) = \rho^{\dist(G)+1}(G)$ when $\rho^{\dist(G)}(G) \neq \fix(G)$?

\subsection{Tools}

 Throughout the paper, we will make use of the following ideas and definitions.

The \emph{paint cost} of $d$-distinguishing $G$ is defined to be the minimum size of the complement of a color class over all $d$-distinguishing colorings. Equivalently, we can determine the paint cost by identifying a largest color class over all $d$-distinguishing colorings and subtracting the size of that class from the total number of vertices in the graph. When considering the definition from this latter perspective, we will differentiate between vertices in the largest color class by assigning to them a \emph{neutral color}, denoted $c_0$, and to all other vertices a \emph{special color}, denoted by a color in the set $\{c_1, \dots, c_{d-1}\}$ in a $d$-distinguishing coloring.

Let $\Gamma(G)$ denote the automorphism group of a graph $G$ and for a vertex $v \in V(G)$, we say the \emph{orbit} of $v$ under $\Gamma(G)$ is the set $\{\phi(v) : \phi \in \Gamma(G)\}$. Since distance is preserved by automorphisms, any orbit in a perfect $k$-ary tree will consist of vertices at the same depth. 

For vertex-colored graphs, we single out two types of graph isomorphisms: A graph isomorphism between two vertex colored graphs $%
G_{1},G_{2}$ is called a \emph{weak isomorphism }if two vertices of $G_{1}$
are the same color if and only if their images in $G_{2}$ are the same
color, and is called a \emph{strong isomorphism} if the vertex color sets
of $G_{1}$ and $G_{2}$ are identical and the isomorphism is color-preserving.

Throughout the paper, we will make use of the following definition: For $0\le i \le n$, a \emph{leafy} subtree $T_k^i$ of $T_k^n$ is a copy of a perfect $k$-ary subtree of depth $n$ that forms a subgraph of $T_k^n$ such that all the leaves of $T_k^i$ are also leaves in $T_k^n$. This definition is useful as perfect $k$-ary trees have a natural recursive structure where a tree of depth $n$ is formed by connecting the root vertex to the roots of $k$ perfect $k$-ary trees of depth $n-1$. 

We also observe that since $\tnk$ is a uniquely rooted tree, any automorphism of $\tnk$ is determined by its action on leaves. It is easily seen that for any automorphism $\phi$ of $\tnk$, $\phi$ either permutes the leaves of a leafy subtree of depth 1 or maps those leaves bijectively to the leaves of another leafy subtree of depth 1.

Finally, we highlight the following results of Erwin and Harary~\cite{EH2006}, which we will use throughout the paper: (1) For every tree $T$, there is a fixing set of $T$ of size $\fix(T)$ consisting only of leaves of $T$; (2) Let $T$ be a tree and $S \subset V(T)$. Then $S$ fixes $T$ if and only if $S$ fixes the leaves of $T$.

\section{Distinguishing numbers of perfect $k$-ary trees}\label{sec-dist}

In this section, we investigate the distinguishing number for perfect $k$-ary trees and the paint cost when $d = \dist(T_k^n)$. In particular, we show that the distinguishing number of $\tnk$ is $k$ and that the $k$-paint cost is $k^n-1$.

\begin{theorem}\label{dist}
    The distinguishing number of a perfect $k$-ary tree $T_k^n$ is $k$ and the $k$-paint cost $\rho^k(T_k^n)$ is $k^n - 1$.

\end{theorem}

\begin{proof}

Note that each level of depth $i$, $0 \leq i \leq n$, of $T_k^n$ has $k^i$ vertices and observe that $k^n-1 = \frac{k-1}{k} ( \frac{k^{n+1}-1}{k-1} - 1) = \frac{k-1}{k} \sum_{i=1}^n  k^i$, which is equivalent to taking $\frac{k-1}{k}$ of the vertices on each level, minus the root vertex at level 0. We let the neutral color be the largest color class and denote it by $c_0$.

We prove this claim by recursion.  For the base case of $n=1$, the $k$ leaves of the perfect $k$-ary tree $T_k^1$ must be colored with distinct colors $c_0, \ldots, c_{k-1}$, and the root can be colored with the neutral color $c_0$.
So the distinguishing number of $T_k^1$ is $k$ and the $k$-paint cost is $k-1$.

For the recursive step, suppose that the distinguishing number of $T_k^{n-1}$ is $k$, with the root colored the neutral color and the $k$-paint cost equal to $\frac{k-1}{k} \sum_{i=1}^{n-1}  k^i$. The perfect $k$-ary tree $T_k^n$ contains leafy subtrees of depth 1, so any coloring that removes symmetries must use at least $k$ colors.  $T_k^n$ is composed of a root connected to the roots of $k$ leafy subtrees $T_k^{n-1}$.  Use the $k$-distinguising coloring on each of those leafy subtrees, then distinguish the subtrees by coloring their roots with the $k$ colors $c_0, \ldots, c_{k-1}$.  (See Figure~1 for an example for $T_3^3$.)  Since $k$ colors suffice to remove all symmetries, the distinguishing number of $T_k^n$ is $k$.  Also, arguing recursively starting with the leaves, we've constructed the unique (up to strong isomorphism) $k$-distinguishing coloring of $T_k^n$ with neutral root, and it follows that the $k$-paint cost is:

	\begin{align*}
		(k-1) + k  \left(\frac{k-1}{k} \sum_{i=1}^{n-1} k^i\right) & = 
        (k-1) + \frac{k-1}{k} \sum_{i=2}^{n} k^i\\
        &= \frac{k-1}{k} k + \frac{k-1}{k} \sum_{i=2}^{n} k^i\\
        &= \frac{k-1}{k} \sum_{i=1}^{n} k^i.
	\end{align*}
    
\end{proof}

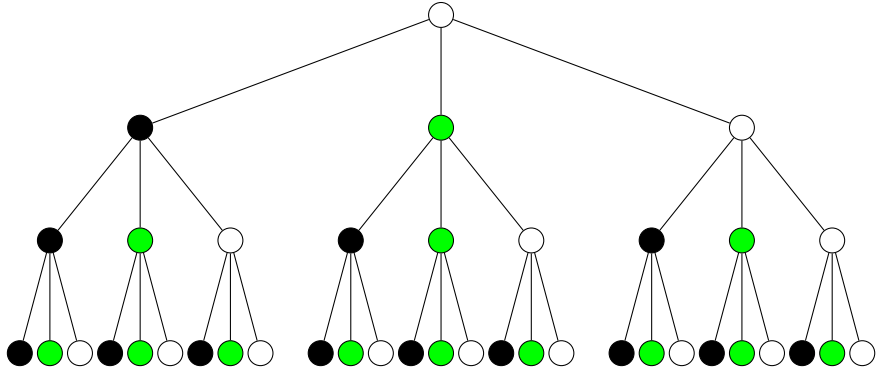
\begin{figure}
\centering

\begin{tikzpicture}[
  every node/.style={draw,circle,minimum size=0.3cm},
  level 1/.style={sibling distance=4cm},
  level 2/.style={sibling distance=1.2cm},
  level 3/.style={sibling distance=0.4cm},
]
\node {}
  child {node [fill=black] {}
    child {node  [fill=black]  {}
      child {node [fill=black] {}}
      child {node [fill=green] {}}
      child {node [fill=white] {}}
    }
    child {node [fill=green]  {}
      child {node [fill=black] {}}
      child {node [fill=green] {}}
      child {node [fill=white] {}}
    }
    child {node {}
      child {node [fill=black] {}}
      child {node [fill=green] {}}
      child {node [fill=white] {}}
    }
  }
  child {node [fill=green] {}
    child {node  [fill=black]  {}
      child {node [fill=black] {}}
      child {node [fill=green] {}}
      child {node [fill=white] {}}
    }
    child {node [fill=green]  {}
      child {node [fill=black] {}}
      child {node [fill=green] {}}
      child {node [fill=white] {}}
    }
    child {node {}
      child {node [fill=black] {}}
      child {node [fill=green] {}}
      child {node [fill=white] {}}
    }
  }
  child {node {}
    child {node [fill=black]  {}
      child {node [fill=black] {}}
      child {node [fill=green] {}}
      child {node [fill=white] {}}
    }
    child {node [fill=green]  {}
      child {node [fill=black] {}}
      child {node [fill=green] {}}
      child {node [fill=white] {}}
    }
    child {node {}
      child {node [fill=black] {}}
      child {node [fill=green] {}}
      child {node [fill=white] {}}
    }
  };
\end{tikzpicture}
\caption{3-distinguishing coloring of a perfect trinary tree.}
\end{figure}

\section{Fixing and frugal numbers of perfect $k$-ary trees}\label{sec-fix}

Recall that the frugal distinguishing number of a graph $G$, denoted $\fdist(G)$ is the smallest $d$ for which $\rho^d(G) = \det(G)$. In this section we determine the fixing number and frugal number for perfect $k$-ary trees.

We say a $d$-distinguishing coloring of $T_{k}^{n}$ is \emph{efficient} if among all $d$-distinguishing colorings it has a maximal number of neutrally colored vertices, \textit{i.e}. it realizes the $d$-paint cost. We note that any frugal coloring of $T_{k}^{n}$ is efficient as a coloring is \emph{frugal} if the complement of the set of neutrally colored
vertices has cardinality $\fix(T_{k}^{n})$ and there is no such $%
(d-1)$-distinguishing coloring.

\begin{prop}\label{fix}
    The fixing number of a perfect $k$-ary tree of depth $n$ is $(k-1)~\cdot~k^{n-1}$.
\end{prop}

\begin{proof}
The transposition of two leaves of a leafy subtree of depth 1 is an automorphism of $\tnk$. Therefore, for each leafy subtree of depth 1, all but one of the leaves must be in any minimal fixing set. In a perfect $k$-ary tree of depth $n$, there are $k^{n-1}$ leafy subtrees of depth 1. Therefore, $(k-1)\cdot k^{n-1}$ vertices must be fixed. Since this set of leaves fixes all of the leaves, it is a minimal fixing set of $\tnk$.
\end{proof}

We next present two lemmas which we will make use of in Theorem~\ref{FrugalN}, which proves the frugal number of $T_k^n$.

\begin{lemma}\label{PermLemma}
Let $a_{1},\cdots ,a_{k}$ be $k\geq 3$ distinct real numbers (or elements of
an ordered abelian group). Then there exist distinct permutations $\sigma
_{1},\cdots ,\sigma _{k}\in S_{k}$ such that the $k$ sums $s_{j}=\sum\limits_{i=1}^{k}a_{\sigma _{i}(j)},\ 1\leq j\leq k$ are distinct. In matrix terms, the matrix $\left( a_{\sigma _{i}(j)}\right) $ will have
distinct column sums.
\end{lemma}

\begin{proof}
Reorder  $a_{1},\cdots ,a_{k}$ so that $a_{1}>\cdots >a_{k}$. \ Then set $%
\sigma _{1}=id$ and for $2\leq i\leq k$ set $\sigma _{i}$ to be the $i$
cycle $(k,k-1,\cdots ,k-i+1)$. \ Then we have%
\[
a_{\sigma _{i}(j)}=\left\{ 
\begin{array}{cc}
a_{j}, & 1\leq j\leq k-i \\ 
a_{k}, & j=k-i+1 \\ 
a_{j-1} & k-i+1<j\leq k%
\end{array}%
\right. \text{ or equivalently }\left\{ 
\begin{array}{cc}
a_{j}, & 1\leq i\leq k-j \\ 
a_{k}, & i=k-j+1 \\ 
a_{j-1} & k-j+1<i\leq k%
\end{array}%
\right. 
\]%
It follows then (setting, say, $a_{0}=0$ when $j=1$) that  for $1\leq j\leq
k,\ \ s_{j}=(k-j)a_{j}+a_{k}+(j-1)a_{j-1}$, and for $1\leq j\leq k-1,\ \
s_{j+1}=(k-j-1)a_{j+1}+a_{k}+ja_{j}.$ \  Subtracting, we find that for $%
1\leq j\leq k-1:$ 
\[
s_{j}-s_{j+1}=(k-j-1)(a_{j}-a_{j+1})+(j-1)(a_{j-1}-a_{j})>0.
\]%
(The coefficients $k-j-1$ and $j-1$ cannot both be 0 as this would force $k=2
$.) \ So in fact, $s_{1}>s_{2}>\cdots >s_{k}$.
\end{proof}

\bigskip

\begin{example}
It may help to look at a special case, for example $k=5$. \ Then, with $%
a_{1}>a_{2}>a_{3}>a_{4}>a_{5}$ we have the matrix with permuted rows:%
\[
\left( 
\begin{array}{ccccc}
a_{1} & a_{2} & a_{3} & a_{4} & a_{5} \\ 
a_{1} & a_{2} & a_{3} & a_{5} & a_{4} \\ 
a_{1} & a_{2} & a_{5} & a_{3} & a_{4} \\ 
a_{1} & a_{5} & a_{2} & a_{3} & a_{4} \\ 
a_{5} & a_{1} & a_{2} & a_{3} & a_{4}%
\end{array}%
\right) 
\]

$s_{1}=4a_{1}+a_{5}$

$s_{2}=3a_{2}+a_{5}+1a_{1\ \ \ }\qquad \qquad s_{1}-s_{2}=3(a_{1}-a_{2})>0$

$s_{3}=2a_{3}+a_{5}+2a_{2}\qquad \qquad \ \ \
s_{2}-s_{3}=2(a_{2}-a_{3})+1(a_{1}-a_{2})>0$

$s_{4}=1a_{4}+a_{5}+3a_{3}\qquad \qquad \ \ \
s_{3}-s_{4}=1(a_{3}-a_{4})+2(a_{2}-a_{3})>0$

$s_{5}=0a_{5}+a_{5}+4a_{4}\qquad \qquad \ \ \
s_{4}-s_{5}=0(a_{4}-a_{5})+3(a_{3}-a_{4})>0$

\bigskip

\end{example}

 While the lemma above suffices for the purposes of Theorem~\ref{FrugalN}, it can be substantially generalized.  In fact, given $k$ distinct elements of an ordered abelian group, it is possible to produce any number up to $k!-1$ permutations with distinct column sums. We include the statement and proof of this fact in Section~\ref{biglemma}.

A further lemma is useful for establishing efficiency recursively in the proof of Theorem~\ref{FrugalN}:

\begin{lemma}\label{efficiency} \ If there are at least $k$ distinct (up to strong
isomorphism) efficient $d$-distinguishing colorings of $T_{k}^{n-1}$, then 

(1) any efficient $d$-distinguishing coloring of $T_{k}^{n}$ restricts to an
efficient $d$-distinguishing coloring on each leafy subtree of depth $n-1$,
and

(2) if the root of $T_{k}^{n}$ is given the neutral color and each leafy
subtree of depth $n-1$ is given a distinct (up to strong isomorphism)
efficient $d$-distinguishing coloring, the resulting coloring on $T_{k}^{n}$
is efficient.
\end{lemma}

\begin{proof} \ (1): \ If some leafy subtree of depth $n-1$ is not
efficiently colored, we can replace it with one of our $k$ distinct (up to
strong isomorphism) efficient $d$-distinguishing colorings of $T_{k}^{n-1}$,
chosen to be not strongly isomorphic with any of the other leafy subtrees at
that depth. \ This would generate a more efficient coloring.

(2): \ The root of any given efficient coloring on $T_{k}^{n}$ must be
neutral, and each leafy subtree of depth $n-1$ must be efficiently colored
by (1), so the constructed coloring must have the same number of neutrally
colored vertices as the given one.
\end{proof}

\begin{theorem}\label{FrugalN}
    The frugal number of a perfect $k$-ary tree of depth $n$ is $k+2$, for $n\ge 3$.  For $n=1$, the frugal number is $k$; for $n=2$, the frugal number is $k+1$.
\end{theorem}

\begin{proof}
Denote the colors as $c_0,c_1,\ldots, c_k$, with $c_0$ indicating the neutral color, and the rest indicating the special colors. We prove the result for $(1)$ $ n=1$, $(2)~$$n=2$, $(3)$ $ n \geq 3$ and $k\geq 3$, and then finally $(4)$ $ n \geq 3$ and $k=2$, each in turn.

For depth $n=1$, we must color each of the leaves of $T_k^1$ with a unique color to remove all symmetries.  Frugal number includes the neutral color and is therefore $k$.

For depth $n=2$, each of the leaves of each leafy subtree $T_k^1$ of depth 1 must be differently colored, with one neutral color to minimize the use of the special colors (thus coloring with special colors only the nodes that must be fixed).  Each leafy subtree $T_k^1$ must therefore have one neutrally-colored leaf and the rest of its $k-1$ leaves must be colored with a distinct combination of the $k$ special colors.  Fortunately, there are exactly $\binom{k}{k-1} = k$ such combinations, exactly the number of distinct leafy subtrees of depth 1.

For $k \ge 3$, we will prove the claim for $n \ge 3$ by recursion, returning to the case of perfect binary trees later. The following method efficiently colors the leaves of a perfect $k$-ary tree at depth $n=3$ in such a way that the counts of the $k+2$ colors on the leaves are all distinct.

As shown above, there is a frugal coloring of the leaves of a depth-2 perfect $k$-ary tree using $k+1$ colors $c_0,c_1,\ldots, c_k$, using all possible combinations of the special colors $c_1,\ldots, c_k$ on its leafy subtrees of depth 1.  The neutral color is used on $k$ leaves and each of the special colors color $k-1$ leaves. As $T^3_k$ contains $k$ distinct copies of leafy subtree $T^2_k$, the frugal number for $T^3_k$ must be greater than $k+1$.

Number the $k$ distinct copies of leafy subtree $T^2_k$ in $T^3_k$.  Color the leaves of each copy with the colors $c_0,c_1,\ldots, c_k$ in the same way as the frugal coloring of a depth-2 perfect $k$-ary tree, then do the following color substitutions:

\begin{itemize}
    \item For the second copy, substitute one $c_1$ with $c_{k+1}$.
    \item For the third copy, substitute two $c_2$ with $c_{k+1}$.
    \item For the fourth copy, substitute three $c_3$ with $c_{k+1}$.

    $\cdots$

    \item For the $k$-th copy, substitute $k-1$ $c_{k-1}$ with $c_{k+1}$.
\end{itemize}

This coloring guarantees that no two copies of leafy subtree $T_k^2$ have the same coloring, so the only automorphism on $T_k^3$ is the trivial one.  Figure~2 shows an example of such a coloring for a depth-3 trinary perfect tree.
The color counts on the leaves of the tree are as follows:

\begin{itemize}
    \item the neutral color $c_0$ colors $k^2$ leaves
    \item $c_i$ colors $(k-1)^2 + (k-i-1)$ leaves, for $i=1, \cdots, k-1$
    \item $c_{k}$ colors $(k-1)k$ leaves
    \item $c_{k+1}$ colors $(k-1)k/2$ leaves
\end{itemize}

For $k \ge 3$, observe that all counts of different colors are distinct.

For $n > 3$ and $k \ge 3$,  the recursive step is as follows.  Suppose that for $T_k^{n-1}$, there is a frugal coloring using $k+2$ colors in such a way that internal nodes are colored by the neutral color $c_0$, and the leaves are colored by colors $c_0, \ldots, c_{k+1}$ in such a way that the number of leaves colored by each color $c_1, \ldots, c_k$ are all unique.  

The perfect $k$-ary tree $T_k^n$ has $k$ distinct leafy subtrees $T_k^{n-1}$.  We color the leaves of one of these subtrees with the above frugal coloring that we supposed to exist.  We color the rest of the leafy subtrees $T_k^{n-1}$ by using the same kind of coloring, but with permuted special colors $c_1, \ldots, c_k$, thus making each leafy subtree distinct from the other.  We need to find $k$ distinct permutations of $k$ colors such that the total number of leaves in the resulting $T_k^n$ colored by each color $c_1, \ldots, c_k$ is unique.  By Lemma~\ref{PermLemma}, there exist $k$ such permutations, completing the recursive step.  Lemma~\ref{efficiency} applies in the recursion to guarantee the coloring is efficient, and must in fact be frugal since the number of colors used has not changed in the recursive step.

For binary trees with $n \ge 3$, the recursive proof that the frugal number is 4 is slightly modified from the recursion above. Observe that each perfect binary tree $T_2^n$ has two leafy subtrees $T_2^{n-1}$, which we arbitrarily call `left' and `right' and observe further that this implies the frugal number of $T^n_2$, $n\geq 3$, is at least 4.

For $n = 3$, we will use the same coloring as above: color the left leafy subtree same as above, and for the right leafy subtree substitude color $c_3$ for $c_1$, again coloring the root with the neutral color. The number of leaves colored with $c_1$ and $c_3$ is 1, and the number of leaves colored with $c_2$ is 2.

For the recursive step, suppose that we have a frugal coloring of $T_2^{n-1}$. As the number of leaves in $T_2^{n-1}$ is $2^{n-1}$ with $2^{n-2}$ leaves receiving special colors, it follows that there exist two special colors for which the number of leaves painted with them is not equal. The following procedure then provides a frugal coloring for $T_2^n$: color the root with the neutral color $c_0$; color the left leafy subtree $T_2^{n-1}$ with the frugal coloring; color the right leafy subtree $T_2^{n-1}$ with a similar coloring, but switching two of the colors whose leaf counts are not equal. The resulting coloring is frugal and uses the same four colors.

Lemma~\ref{efficiency} again guarantees the efficiency and, in fact, frugality of the coloring.

\end{proof}

\begin{figure}\label{five-color}
\centering

\begin{tikzpicture}[
  every node/.style={draw,circle,minimum size=0.3cm},
  level 1/.style={sibling distance=4cm},
  level 2/.style={sibling distance=1.2cm},
  level 3/.style={sibling distance=0.4cm},
]
\node {}
  child {node {}
    child {node {}
      child {node [fill=black] {}}
      child {node [fill=green] {}}
      child {node [fill=white] {}}
    }
    child {node {}
      child {node [fill=black] {}}
      child {node [fill=cyan] {}}
      child {node [fill=white] {}}
    }
    child {node {}
      child {node [fill=green] {}}
      child {node [fill=cyan] {}}
      child {node [fill=white] {}}
    }
  }
  child {node {}
    child {node {}
      child {node [fill=red] {}}
      child {node [fill=green] {}}
      child {node [fill=white] {}}
    }
    child {node {}
      child {node [fill=black] {}}
      child {node [fill=cyan] {}}
      child {node [fill=white] {}}
    }
    child {node {}
      child {node [fill=green] {}}
      child {node [fill=cyan] {}}
      child {node [fill=white] {}}
    }
  }
  child {node {}
    child {node {}
      child {node [fill=black] {}}
      child {node [fill=red] {}}
      child {node [fill=white] {}}
    }
    child {node {}
      child {node [fill=black] {}}
      child {node [fill=cyan] {}}
      child {node [fill=white] {}}
    }
    child {node {}
      child {node [fill=red] {}}
      child {node [fill=cyan] {}}
      child {node [fill=white] {}}
    }
  };
\end{tikzpicture}
\caption{Frugal coloring of depth-3 trinary perfect tree, using five colors: white (neutral), black, green, cyan, and red.}
\end{figure}
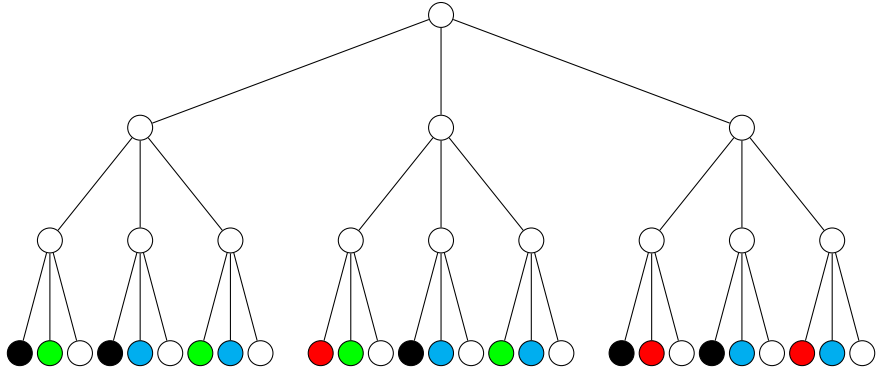

The following is a direct corollary of Theorem~\ref{FrugalN} and Proposition~\ref{fix}:

\begin{coroll}
    The paint cost of $(k+2)$-distinguishing a perfect $k$-ary tree of depth $n\ge 3$ is $\rho^{k+2}(T_k^n) = (k-1)k^{n-1}$.
\end{coroll}

\section{Paint cost for $k+1$ colors}\label{sec-middle}

In this section, we determine the $(k+1)$-paint cost of perfect $k$-ary trees. We then combine all results into Theorem~\ref{pc}, summarizing the paint cost spectrum and paint cost ratio for the full class of perfect $k$-ary trees.

\begin{theorem}\label{middle}
    The paint cost $\rho^{k+1}(T_k^n)$ of $(k+1)$-distinguishing a perfect $k$-ary tree of depth $n$ is $k-1$ for $n=1$, $k(k-1)$ for $n=2$, and $(k-1)(k^2+1)k^{n-3}$ for $n\ge 3$.  
\end{theorem}

\begin{proof}
    Let $c_0$ indicate the neutral color and $c_1, \ldots, c_k$ indicate the special colors.

    For a perfect $k$-ary tree $T_k^1$ of depth 1, all the $k$ leaves must be differently colored to remove symmetries.  For the most frugal use of special colors, color the root and one of the leaves with the neutral color $c_0$, and color the rest of with a combination of special colors $c_1, \ldots, c_k$.  So $\rho^{k+1}(T_k^1) = k-1$.

    A perfect $k$-ary tree $T_k^2$ of depth 2 has $k$ copies of leafy subtrees $T_k^1$.  Within each leafy subtree $T_k^1$, all the $k$ leaves must be differently colored to remove symmetries. Thus, the minimum possible cost is $k(k-1)$. Color one leaf from each leafy subtree the neutral color $c_0$, along with its root and the root of $T_k^2$ itself. We wish to distinguish the leafy subtrees $T_k^1$ without coloring any more vertices; we do so by using different combinations of $k-1$ colors chosen out of a total of $k$ special colors.  Since there are exactly $k$ such combinations, there is exactly one way to achieve this coloring.  So $\rho^{k+1}(T_k^2) =k(k-1)$.

A perfect $k$-ary tree $T_{k}^{3}$ of depth 3 has $k$ copies of leafy
subtrees $T_{k}^{2}$. We develop a coloring of $T_{k}^{3}$ that uses special
colors efficiently by coloring all leafy subtrees $T_{k}^{2}$ with the
above-described coloring and then considering how to break symmetries
between them. Call a $k$-distinguishing coloring of a tree
``almost-efficient" if it has exactly one fewer neutrally colored vertex than
an efficient one.  Since the efficient coloring of $T_{k}^{2}$ is unique up
to strong isomorphism, one cannot produce a more efficient coloring of $%
T_{k}^{3}$ than one obtained by recoloring one neutrally-colored vertex in $%
k-1$ of the leafy subtrees $T_{k}^{2}$. \ We show below that this can be
done (multiple ways), making the paint cost $\rho
_{k+1}(T_{k}^{3})=(k-1)k^{2}+k-1=(k-1)(k^{2}+1).$

We observe  that there are $k^{2}+2k$ ways to create non-isomorphic
almost-efficient colorings of $T_{k}^{2}$ as above: $k$ choices for coloring
the root of $T_{k}^{2}$ with a special color; $k^{2}$ choices for coloring a
vertex at depth 1 with a special color; and $k$ ways of coloring a leaf with
a special color. (In case of coloring a leaf: once we pick a leafy subtree $%
T_{k}^{1}$, there is a single neutrally-colored leaf, and a single special
color that's different from the colors on other leaves.) Since $%
(k+2)(k-1)<k^{2}+2k$, we can easily create $k+2$ distinct (up to strong
isomorphism) efficient $(k+1)$-colorings of $T_{k}^{3}$.  Do so, and call
them $a_{0},a_{1},\ldots ,a_{k+1}.$

\bigskip 

For $n>3$, we find a $(k+1)$-distinguishing coloring of $T_{k}^{n}$ as
follows. First, we find a $(k+2)$-distinguishing coloring of $T_{k}^{n-3}$
which uses special colors only on the fixing set consisting of $(k-1)k^{n-1}$
leaves. By Theorem~\ref{FrugalN}, we know that such coloring exists.  We then substitute the leaves colored by $c_0, c_1, \ldots, c_{k+1}$ by leafy subtrees $T_k^3$ with colorings $a_0, a_1, \ldots, a_{k+1}$ that have corresponding indexes. Since all
leafy subtrees $T_{k}^{3}$ are efficiently colored with $k+1$ colors and no
other interior vertex is colored with special colors, we can, by induction
on depth, apply Lemma~\ref{efficiency} to conclude this coloring of $T_{k}^{n}$ is the most
efficient possible. Therefore, $\rho ^{k+1}(T_{k}^{n})=(k-1)(k^{2}+1)k^{n-3}$%
. 
    
\end{proof}

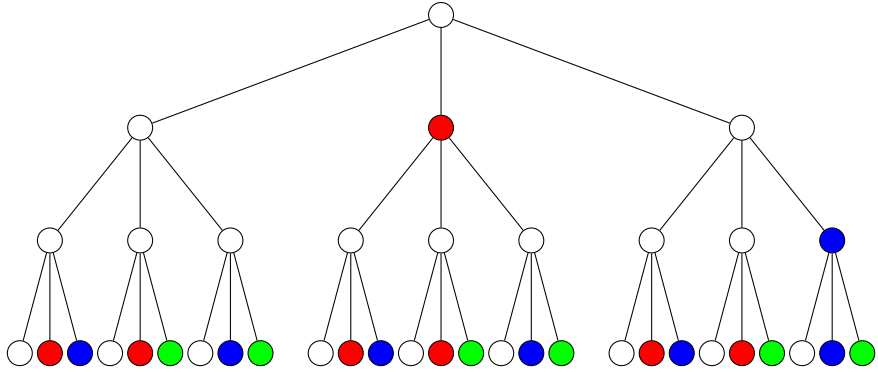
\begin{figure}
\centering

\begin{tikzpicture}[
  every node/.style={draw,circle,minimum size=0.3cm},
  level 1/.style={sibling distance=4cm},
  level 2/.style={sibling distance=1.2cm},
  level 3/.style={sibling distance=0.4cm},
]
\node {}
  child {node [fill=white] {}
    child {node  [fill=white]  {}
      child {node [fill=white] {}}
      child {node [fill=red] {}}
      child {node [fill=blue] {}}
    }
    child {node [fill=white]  {}
      child {node [fill=white] {}}
      child {node [fill=red] {}}
      child {node [fill=green] {}}
    }
    child {node {}
      child {node [fill=white] {}}
      child {node [fill=blue] {}}
      child {node [fill=green] {}}
    }
  }
  child {node [fill=red] {}
    child {node  [fill=white]  {}
      child {node [fill=white] {}}
      child {node [fill=red] {}}
      child {node [fill=blue] {}}
    }
    child {node [fill=white]  {}
      child {node [fill=white] {}}
      child {node [fill=red] {}}
      child {node [fill=green] {}}
          }
    child {node {}
     child {node [fill=white] {}}
      child {node [fill=blue] {}}
      child {node [fill=green] {}}
    }
  }
  child {node {}
     child {node  [fill=white]  {}
      child {node [fill=white] {}}
      child {node [fill=red] {}}
      child {node [fill=blue] {}}
    }
    child {node [fill=white]  {}
      child {node [fill=white] {}}
      child {node [fill=red] {}}
      child {node [fill=green] {}}
    }
    child {node [fill=blue] {}
      child {node [fill=white] {}}
      child {node [fill=blue] {}}
      child {node [fill=green] {}}
    }
  };
\end{tikzpicture}
\caption{One of many efficient 4-colorings of $T^3_3$.}
\end{figure}

Combining the results of Theorems~\ref{dist}, \ref{middle}, and~\ref{FrugalN}, as well as Proposition~\ref{fix}, we can completely determine the paint cost spectrum and paint cost ratio for perfect $k$-ary trees.

\bigskip

\begin{theorem}\label{pc}
Let $T^n_k$ be a perfect $k$-ary tree of depth $n$ for $n \geq 1$. Then, 
\begin{itemize}
    \item for $n=1,$ the paint cost spectrum is $(k;k-1)$ and the paint cost ratio is $1$,
    \item for $n=2,$ the paint cost spectrum is $(k; k^2-1, k(k-1), \dots, k(k-1))$ and the paint cost ratio is $\displaystyle \frac{k^2-2k+1}{k^2-2k+2}$,
    \item for $n \geq 3$, the paint cost spectrum is $$(k; k^n-1, (k-1)(k^2+1)k^{n-3}, k^{n-1}(k-1), \dots, k^{n-1}(k-1))$$ and the paint cost ratio is $ \displaystyle \frac{k^{n-1}(k-1)-k}{k^{n-1}(k-1) - k+2}$.
\end{itemize}
    
\end{theorem}

Lastly, we observe here that $T^n_k$ is defined only for $k\geq 2$ and $n\geq 1$, but that if we allowed $k=1$, $T^n_1$ would result in a path graph with $n$ edges. For completeness, we note that such path graphs are 2-distinguishable with paint cost 1 by coloring one vertex at the end of the path a special color and coloring all other vertices a neutral color. Similarly, path graphs have fixing number 1. Thus, we can say that the paint cost spectrum for a path graph is $(2; 1)$ and has paint cost ratio of 1.

\bigskip

\section{Equality on the Paint Cost Spectrum}\label{equality}

The paint cost spectrum of a graph $G$ has equality for all paint costs when the paint cost of distinguishing $G$ with $\dist(G)$ colors is equal to the $\fdist(G)$-paint cost; this occurs when the paint cost ratio is equal to 1. In~\cite{B2023}, Boutin asked which graphs have this property and in this section, we provide several examples. We also show that for any paint cost greater than or equal to 1, there exists an infinite class of graphs with equality on the paint cost spectrum.

\begin{example}
    Graphs that are asymmetric (their only automorphism is the trivial automorphism) have equality on the paint cost spectrum. Let $G$ be an asymmetric graph. Then $\dist(G) = 1$ and $\det(G) = 0$ and it follows that $\rho^1(G) = \rho^{\det(G)} = 0$ so the paint cost spectrum of any asymmetric graph is $(1; 0)$.
\end{example}

\begin{example}
    Complete graphs, $K_n$, have equality on the paint cost spectrum: $\dist(K_n) = n$ and $\rho^{n} = n-1$. Further, $\det(K_n) = n-1$. Thus, the paint cost spectrum of $K_n$ is $(n; n-1)$.
\end{example}

\begin{example}
    The class of complete $r$-partite graphs have equality on the paint cost spectrum. Let $G = K_{k_1, k_2, \dots, k_r}$ be a complete $r$-partite graph with partitions of size $k_1, k_2, \dots, k_r$. In any distinguishing coloring of $G$, each vertex in a partition will need to receive a different color than the other vertices in that same partition. Though colors may be repeated in different partitions if some partitions have unique sizes, the $\dist(G)$-paint cost will still be the sum $(k_1-1) + (k_2-1) + \cdots + (k_r-1)$. Similarly, a fixing set must also fix all but one vertex from each partition and so $\det(G) = (k_1-1) + (k_2-1) + \cdots + (k_r-1) = \rho^{\det(G)+1}(G)$.
\end{example}

\begin{example}
    Let $t_1, t_2, \dots t_k$ be positive integers. A \textbf{$k$-pode} $T_k(t_1, t_2, \dots, t_k)$ is the tree that has one degree $k$ vertex, $v$, the removal of which leaves $k$ disjoint paths having $t_1, t_2, \dots, t_k$ vertices. Let $T = T_k(t_1, t_2, \dots, t_k)$ be a $k$-pode graph and let $\mathcal{A}_1, \mathcal{A}_2, \dots \mathcal{A}_j$, $1 \leq j \leq k$ be the collections of isomorphic paths in $T - v$. For each $\mathcal{A}_i$, let $a_i$ denote the number of vertices on each path in $\mathcal{A}_i$. Since paths of unique length in $T$ are fixed in any automorphism, $\det(T) = \sum_{i=1}^j |\mathcal{A}_i| - 1$. If it is also the case that $|\mathcal{A}_i| \leq a_i$ for all $1 \leq i \leq j$, then $T$ is $2$-distinguishable and the 2-paint cost is precisely $\det(T)$, giving equality throughout the paint cost spectrum. 
\end{example}

The $k$-pode graphs can be used to demonstrate the following theorem.

\begin{theorem}\label{equal}
    Let $d \geq 1$ be a positive integer. Then there exists an infinite class of graphs such that $d$ is equal to the $\dist(G)$-paint cost, which in turn is equal to the $\fix(G)$-paint cost.
\end{theorem}

\begin{proof}
    Let $G$ be a $k$-pode graph, $T_{d+1}(t_1, t_2, \dots t_{d+1})$ such that for $1 \leq i \leq d+1$, all the $t_i$'s are equal and $t_i \geq d$. A fixing set for $G$ consists of $d$ of the $d+1$ leaf vertices. Further, $G$ is 2-distinguishable and we can color the following $d$ vertices a second color: for $1 \leq i \leq d$, color the $i^{th}$ vertex on the path $t_i$.
\end{proof}

\section{A Generalization of Lemma~\ref{PermLemma}}\label{biglemma}

In this section we give a generalization of Lemma~\ref{PermLemma}, which could prove
useful in coloring problems or other contexts.

Let $\overrightarrow{v}=(a_{1},\cdots ,a_{n})$ be an $n$-tuple of distinct
real numbers (or more generally, elements of an ordered abelian group). \ We
call a $k\times n$ matrix $A=(\alpha _{i,j})$ a \emph{$\overrightarrow{v%
}$-row permuted matrix} if the rows of $A$ are distinct permutations of the
components of $\overrightarrow{v}$. \ Such a matrix will be said to satisfy
the \emph{distinct column sum property (DCS)} if its column sums $%
\sum\limits_{i=1}^{k}\alpha _{i,j}$ $(1\leq j\leq n)$ are distinct.

Clearly, to form a $k\times n$ $\overrightarrow{v}$-row permuted matrix we
must have $1\leq k\leq n!$, and for the matrix to additionally satisfy DCS, $%
1\leq k\leq n!-1$.

We call two permutations of $(a_{1},\cdots ,a_{n})$ equivalent if one is a
cyclic permutation of the other. \ An equivalence class will be called a 
\emph{cyclic block}.

\begin{prop}
Given $\overrightarrow{v}=(a_{1},\cdots ,a_{n})$ an $n$-tuple ($n\geq 2$) of
distinct elements of an ordered abelian group, and given $k$ with $1\leq
k\leq n!-1$, there exists a $k\times n$  $\overrightarrow{v}$-row permuted
matrix $A$ satisfying DCS. 
\end{prop}

\begin{proof}
Suppose we have  $r\times n$ and $s\times n\ \overrightarrow{v}$-row
permuted matrices $M_{r}$ and $M_{s}$ where the column sums of $M_{r\text{ }}
$are constant and the column sums of $M_{s}$ are distinct. We then observe: 
\newline
(1) if no rows of $M_{s}$ are rows of $M_{r}$ , we can append the rows of $%
M_{s}$ to $M_{r}$ to form an $(r+s)\times n$ $\overrightarrow{v}$-row
permuted matrix satisfying DCS, and \newline
(2) if all rows of $M_{s}$ are rows of $M_{r}$ , we can remove the rows of $%
M_{s}$ from $M_{r}$ to form an $(r-s)\times n$ $\overrightarrow{v}$-row
permuted matrix satisfying DCS.\newline
In particular, by observation (2), if we have any $s\times n\ 
\overrightarrow{v}$-row permuted matrix $M_{s}$ satisfying DCS, we can
produce a ``complementary" $(n!-s)\times n$ $\overrightarrow{v}$-row permuted
matrix $M_{s}$ satisfying DCS by choosing as rows all permutations of $%
(a_{1},\cdots ,a_{n})$ which are not rows of $M_{s}$. \ (I.e. take as $M_{r}$
the $n!\times n$ matrix whose rows are all permutations of $(a_{1},\cdots
,a_{n})$). \ It therefore suffices to handle the cases $1\leq k\leq n!/2$. \ 

From this point forward, we assume without loss of generality that $%
a_{1}>a_{2}>\cdots >a_{n}$.\newline
First, we note that $n=2$ is trivial, and we give explicit constructions for 
$n=3\ (1\leq k\leq 3!/2=3)$:%
\begin{eqnarray*}
n &=&3,\ k=1:\left( a_{1}\ a_{2}\ a_{3}\right)  \\
n &=&3,k=2:\left( 
\begin{array}{ccc}
a_{1} & a_{2} & a_{3} \\ 
a_{2} & a_{3} & a_{1}%
\end{array}%
\right)  \\
n &=&3,k=3:\left( 
\begin{array}{ccc}
a_{1} & a_{2} & a_{3} \\ 
a_{1} & a_{3} & a_{2} \\ 
a_{3} & a_{1} & a_{2}%
\end{array}%
\right) 
\end{eqnarray*}%
(the $k=3$ case handled as in Lemma 1). $\ $We henceforth assume $n\geq 4$.
Let $k$ be given, $1\leq k\leq n!/2$. \ By the division algorithm, but
taking remainders from $1$ to $n$ rather than $0$ to $n-1$, write $k=nq+s,\
1\leq s\leq n,\ 0\leq q\leq (n-1)!/2-1$. \ \ We first form an $nq\times n$
matrix $M_{nq}$ with constant column sums by choosing $q$ cyclic blocks of
permutations of $(a_{1},\cdots ,a_{n})$ as rows. \ The choice of cyclic
blocks will depend on $s$, as seen below. \ Our intention is to use
observation (1) above to append an $s\times n$ $\overrightarrow{v}$-row
permuted matrix $M_{s}$ satisfying DCS to $M_{nq}$ to build the desired $%
k\times n$ matrix. Following Lemma 1 for the case $s>2$ (leaving $a_{j}$ in
place for $j>s$), we take $M_{s}$ as follows:%
\begin{eqnarray*}
\text{if }s &=&2,M_{s}=\left( 
\begin{array}{cccccc}
a_{1} & a_{2} & a_{3} & a_{4} & \cdots  & a_{n} \\ 
a_{2} & a_{3} & a_{1} & a_{4} & \cdots  & a_{n}%
\end{array}%
\right)  \\
\text{if }s &>&2,M_{s}=\left( 
\begin{array}{cccccccccc}
a_{1} & a_{2} & a_{3} & \cdots  & a_{s-1} & a_{s} & \cdots  & a_{n-2} & 
a_{n-1} & a_{n} \\ 
a_{1} & a_{2} & a_{3} & \cdots  & a_{s} & a_{s-1} & \cdots  & a_{n-2} & 
a_{n-1} & a_{n} \\ 
\vdots  & \vdots  & \vdots  &  &  &  &  & \vdots  & \vdots  & \vdots  \\ 
a_{1} & a_{s} & a_{2} & \cdots  & a_{s-2} & a_{s-1} & \cdots  & a_{n-2} & 
a_{n-1} & a_{n} \\ 
a_{s} & a_{1} & a_{2} & \cdots  & a_{s-2} & a_{s-1} & \cdots  & a_{n-2} & 
a_{n-1} & a_{n}%
\end{array}%
\right) 
\end{eqnarray*}%

By inspection for $s=2$ and by Lemma 1 for $s>2$, we see that these do
indeed construct $s\times n$ $\overrightarrow{v}$-row permuted matrices $%
M_{s}$ satisfying DCS. \ We see now that, to finish the proof, we need to be
able to make our choices of cyclic blocks in $M_{nq}$ so as to avoid the
permutations in $M_{s}$. \ That means we must avoid at most $n$ cyclic
blocks in our choices. \ Since $q\leq (n-1)!/2-1$, we then need to be able
to choose at most $(n-1)!/2-1$ out of a total of $(n-1)!$ cyclic blocks
while avoiding at most $n$ cyclic blocks. \ That is possible so long as $%
(n-1)!\geq (n-1)!/2-1+n$. \ But that inequality holds precisely for $n\geq 4$%
.
\end{proof}

\begin{example}

Take $n=5,k=18.$ \ For $M_{s}=M_{3}$ we build 

$$\left( 
\begin{array}{ccccc}
a_{1} & a_{2} & a_{3} & a_{4} & a_{5} \\ 
a_{1} & a_{3} & a_{2} & a_{4} & a_{5} \\ 
a_{3} & a_{1} & a_{2} & a_{4} & a_{5}%
\end{array}%
\right) .$$ \ 
To construct $M_{nq}=M_{5\cdot 3}$ we need to choose $3$ cyclic
blocks that avoid the equivalence classes of the rows of $M_{3}$. \ So as a
final product, we could take, for instance (transposed for convenience):%

{\scriptsize
\[
\left( 
\begin{array}{cccccccccccccccccc}
a_{2} & a_{5} & a_{1} & a_{4} & a_{3} & a_{2} & a_{1} & a_{5} & a_{3} & a_{4}
& a_{1} & a_{3} & a_{5} & a_{2} & a_{4} & a_{1} & a_{1} & a_{3} \\ 
a_{3} & a_{2} & a_{5} & a_{1} & a_{4} & a_{4} & a_{2} & a_{1} & a_{5} & a_{3}
& a_{4} & a_{1} & a_{3} & a_{5} & a_{2} & a_{2} & a_{3} & a_{1} \\ 
a_{4} & a_{3} & a_{2} & a_{5} & a_{1} & a_{3} & a_{4} & a_{2} & a_{1} & a_{5}
& a_{2} & a_{4} & a_{1} & a_{3} & a_{5} & a_{3} & a_{2} & a_{2} \\ 
a_{1} & a_{4} & a_{3} & a_{2} & a_{5} & a_{5} & a_{3} & a_{4} & a_{2} & a_{1}
& a_{5} & a_{2} & a_{4} & a_{1} & a_{3} & a_{4} & a_{4} & a_{4} \\ 
a_{5} & a_{1} & a_{4} & a_{3} & a_{2} & a_{1} & a_{5} & a_{3} & a_{4} & a_{2}
& a_{3} & a_{5} & a_{2} & a_{4} & a_{1} & a_{5} & a_{5} & a_{5}%
\end{array}%
\right) ^{T}
\]%
}

$\allowbreak $
If we wanted to handle the case $n=5,k=102$ we would just form
a matrix out of the $102$ permutations not seen in the matrix above.
\end{example}

\section{Future Work and Open Questions}\label{future}

In this section, we share some open problems and directions. Of course, given that the paint cost spectrum and the paint cost ratio are new concepts, determining these parameters for any class of graphs is wide open.\\
\\
Recall that in~\cite{AK21}, Alikhani and Soltani generalized the cost of 2-distinguishing by defining the \emph{cost number} $\rho_d(G)$ (where $d = \dist(G)$) as the size of the smallest color class over all $d$-distinguishing colorings of $G$. We can think of this as answering the question, \textit{how far away is a $d$-distinguishable graph from being $(d-1)$-distinguishable?} We observe here that Theorem~\ref{dist} implies the following result on the cost number of a perfect $k$-ary tree $T^n_k$.

\begin{coroll}\label{costnumber}
The  cost number of a perfect $k$-ary tree $T^n_k$ is given by
\[\rho_k(T^n_k)=\frac{k^n-1}{k-1}.\]
\end{coroll}

We omit the proof of Corollary~\ref{costnumber} as it is just a trivial consequence of Theorem~\ref{dist}. 

The cost number is defined only for the distinguishing number of a graph. However, we can extend this definition in the natural way for a graph $G$ for all $d$ such that $\dist(G) \leq d \leq \fix(G)+1$ and then define a \emph{cost number spectrum} to be the following:
\begin{center}
    $(\dist(G); \rho_{\dist(G)}(G), \rho_{\dist(G)+1}(G), \rho_{\dist(G)+2}(G), \dots, \rho_{\det(G)+1}(G)).$
\end{center}

We observe that for a given graph $G$, $\rho_{\det(G)+1}(G) = 1$. So here, similar to the frugal distinguishing number, the value of interest in the spectrum is the first place in which $\rho_d(G) = 1$. Call this value the \emph{frugal cost number} and denote it by $\fcost(G)$. We can then ask the following questions:
\begin{itemize}
    \item Determine the cost number spectrum for any class of graphs.
    \item How does the frugal cost number relate to the frugal distinguishing number? When are these values equal?
\end{itemize}

In \cite{KA15} and \cite{Mc22}, the authors define edge analogs of the concepts of distinguishing set and fixing set for breaking symmetries of a graph, respectively. Moreover, in \cite{GO19} the authors introduce the idea of the cost of edge-distinguishing a $2$-edge-distinguishable graph. A natural question that arises from combining these concepts is to investigate the edge analog of the paint cost spectrum as defined in this paper. To investigate this spectrum, one may need to first answer the following questions:
\begin{itemize}
    \item Generalize the concept of the cost of edge-distinguishing from just $2$-edge-distinguishable graphs to any $d$-edge-distinguishable graph. As discussed in this paper, there's more than more perspective from which this can be done, for example:
    \begin{itemize}
        \item the minimum number of edges in the graph needing to be recolored from a neutral color in order to have that $d$-edge-distinguishing coloring of $G$,
        \item how far away is $G$ from being $(d-1)$-edge-distinguishable.
    \end{itemize}
    \item Determine the corresponding distinguishing index, fixing index, and frugal index for particular classes of graphs.
\end{itemize}

Lastly, we recall the open question mentioned at the end of Section~\ref{prelim1}:

\begin{itemize}
    \item Determine whether or not it's possible for equal paint costs exist aside from $\rho^d(G)$ where $d \geq \fdist(G)$. In other words, can there be consecutive integers $i$ and $i+1$, with $i+1 < \fdist(G)$, such that $\rho^i(G) = \rho^{i+1}(G)$? And, in particular, is it possible to have $\rho^{\dist(G)}(G) = \rho^{\dist(G)+1}(G)$ when $\rho^{\dist(G)}(G) \neq \fix(G)$?
\end{itemize}

\end{document}